\documentclass[twoside,a4paper,reqno,11pt]{amsart} 
\usepackage{amsfonts, amsbsy, amsmath, amssymb, latexsym}
\usepackage{mathrsfs,array}
\usepackage[top=30mm,right=30mm,bottom=30mm,left=30mm]{geometry}

\usepackage{latexsym}
\usepackage{amssymb}
\usepackage{amsthm}
\usepackage{hyperref}

\newtheorem{thm}{Theorem}[section]
\newtheorem{cor}[thm]{Corollary}
\newtheorem{lem}[thm]{Lemma}

\newtheorem{prop}[thm]{Proposition}

\theoremstyle{definition}
\newtheorem{rem}[thm]{Remark}

\newcommand{\mB}{\mathcal B}

\newcommand{\FF}[1]{\mathbb F_{#1}}

\newcommand{\ra}{\rightarrow}

\newcommand{\fX}{\mathfrak X}

            \def\la{\langle} 
            \def\ra{\rangle} 
          
            \def\a{\alpha}

            \def\O{\Omega} 
             
            \def\l{\lambda}

            \def\F{\mathbb{F}} 
            \def\N{\mathbb{N}} 
            \def\Z{\mathbb{Z}}

            \def\d{\delta} 
      
            \def\go{\rightarrow} 
             
     \def\g{\gamma} 
     \def\G{\Gamma}

\begin{document}

\title{Base sizes of primitive groups: bounds with explicit constants}

\author{Zolt\'an Halasi} \address{Department of Algebra and Number
  Theory, E\"otv\"os University, P\'azm\'any P\'eter s\'et\'any 1/c,
  H-1117, Budapest, Hungary \and Alfr\'ed R\'enyi Institute of
  Mathematics, Hungarian Academy of Sciences, Re\'altanoda utca 13-15, H-1053, Budapest, Hungary}
\email{zhalasi@cs.elte.hu and halasi.zoltan@renyi.mta.hu}

\author{Martin W. Liebeck}
\address{Department of Mathematics, Imperial College, London SW7 2BZ, United Kingdom}
\email{m.liebeck@imperial.ac.uk}

\author{Attila Mar\'oti}
\address{Alfr\'ed R\'enyi Institute of Mathematics, Hungarian Academy of Sciences, Re\'altanoda utca 13-15, H-1053, Budapest, Hungary}
\email{maroti.attila@renyi.mta.hu}

\date{\today}
\keywords{minimal base size, primitive permutation group, classical group, irreducible linear group}

\subjclass[2010]{20B15, 20C99, 20B40.}

\thanks{The work of the first and third authors on the project leading to this application has received funding from the European Research Council (ERC) under the European Union's Horizon 2020 research and innovation programme (grant agreement No 741420, 617747, 648017). The first and third authors were supported by the J\'anos Bolyai Research Scholarship of the Hungarian Academy of Sciences, were funded by a Humboldt Return Fellowship and also by the National Research, Development and Innovation Office (NKFIH) Grant No.~K115799.}
	
\begin{abstract}
We show that the minimal base size $b(G)$ of a finite primitive permutation group $G$ of degree $n$ is at most $2 (\log |G|/\log n) + 24$. This bound is asymptotically best possible since there exists a sequence of primitive permutation groups $G$ of degrees $n$ such that $b(G) = \lfloor 2 (\log |G|/\log n) \rceil - 2$ and $b(G)$ is unbounded. As a corollary we show that a primitive permutation group of degree $n$ that does not contain the alternating group $\mathrm{Alt}(n)$ has a base of size at most $\max\{\sqrt{n} , \ 25\}$.  
\end{abstract}	
\maketitle

\section{Introduction}
Let $G$ be a permutation group acting on a finite set $\Omega$ of size
$n$. A subset $\Sigma$ of $\Omega$ is called a {\it base} for $G$ if the
pointwise stabilizer of $\Sigma$ in $G$ is
trivial. The minimal size of a base for $G$
on $\Omega$ is denoted by $b_{\Omega}(G)$ or by $b(G)$ in case $\Omega$ is clear from the context.

The minimal base size of a primitive permutation group has been much 
investigated. Already in the nineteenth century Bochert \cite{Bochert}
showed that $b(G) \leq n/2$ for a primitive permutation group $G$ of
degree $n$ not containing $\mathrm{Alt}(n)$. This bound was substantially
improved by Babai to $b(G) < 4 \sqrt{n} \log n$, for uniprimitive
groups $G$, in \cite{BabaiAnnals}, and to the estimate $b(G) < 2^{c
  \sqrt{\log n}}$ for a universal constant $c$, for doubly
transitive groups $G$ not containing $\mathrm{Alt}(n)$, in
\cite{BabaiInvent}. (Here and throughout the paper the base of the logarithms is $2$ unless otherwise stated.) 
The latter bound was improved by Pyber
\cite{PyberDoubly} to $b(G) < c {(\log n)}^{2}$ where $c$ is a
universal constant. These estimates are elementary in the sense that
their proofs do not require the Classification of Finite Simple Groups
(CFSG). Using CFSG, Liebeck \cite{Liebeck84} classified all primitive
permutation groups $G$ of degree $n$ with $b(G) \geq 9 \log n$.


It is easy to see that any permutation group $G$ of degree $n$ satisfies $|G| < n^{b(G)}$, and hence 
$b(G) > \log |G| / \log n$. A well-known question of Pyber \cite[Page 207]{pyber}, going back to 1993, asks whether there exists a universal constant $c$ such that $b(G) < c (\log |G| / \log n)$ for all {\it primitive} groups $G$. This question generalizes other conjectures in the area: for example, the Cameron-Kantor conjecture, which asserts that every almost simple primitive group in a non-standard action has base size bounded by a universal constant $C$; and Babai's conjecture, that there is a function $f:\N\go \N$ such that any primitive group that has no alternating or classical composition factor of degree or dimension greater than $d$, has base size less than $f(d)$. 

The Cameron-Kantor conjecture was proved in \cite{LS99} (and in a strong form with $C=7$ in \cite{Burness}, \cite{BurnessLiebeckShalev}). Babai's conjecture was proved in \cite{GSS} with $f$ a quadratic function (improved to a linear function in \cite{LS99}).

Despite a great deal of attention, Pyber's conjecture remained open until very recently, when it was proved in \cite{DHM}. It is shown in \cite{DHM} that there exists a universal constant $c>0$ such that for every primitive permutation group $G$ of degree $n$ we have 
\[
b(G) < 45 (\log |G| / \log n) + c.
\]
To obtain a more explicit, usable bound, one would like to reduce the multiplicative constant 45 in the above, and also to estimate the constant $c$. 

In this paper we achieve this aim. Our main result is the following. 

\begin{thm}
\label{mainresult}
Let $G$ be a primitive permutation group of degree $n$. Then the minimal base size $b(G)$ satisfies
\[
b(G) \le 2 \frac{\log |G|}{\log n} + 24.
\] 
\end{thm}

The multiplicative constant 2 in Theorem \ref{mainresult} is best possible, as is shown by the following. 

\begin{prop}
\label{acc}\leavevmode
\begin{itemize}
\item[{\rm (i)}] For every positive integer $k$ there exists a sequence of finite primitive permutation groups $G_n$ of degrees $n$ such that as $n \to \infty$,
\[
(b(G_n) \log n) / \log |G_n| \to 2k/(k+1).
\]
\item[{\rm (ii)}] There is an infinite sequence of primitive permutation groups $H_n$ of degrees $n$ such that 
$b(H_n) = \lfloor 2 (\log |H_n|/\log n) \rceil - 2$ for all $n$ and $b(H_{n})$ is unbounded. 
\end{itemize}  
\end{prop}

A corollary of Theorem \ref{mainresult} and its proof is the following.

\begin{cor}
\label{maincorollary}
Let $G$ be a primitive permutation group of degree $n$ not containing $\mathrm{Alt}(n)$. Then $G$ has a base of size at most $\max\{\sqrt{n} , \ 25\}$.
\end{cor}

Theorem \ref{mainresult} is proved for almost simple groups in the next two sections (see Theorems \ref{generalalternating}, \ref{generalclassical} and \ref{almostsimple}): alternating and symmetric groups are handled in \S2, and classical groups in \S3. The remaining non-affine primitive groups are covered in \S4 (see Theorem \ref{nonaffine}), and affine groups in \S5 (Theorem \ref{generalaffine}). Proposition \ref{acc} follows from Proposition \ref{2k/(k+1)} and Proposition \ref{sp}. Finally, Corollary \ref{maincorollary} is proved in Section \ref{SecCor}.

\section{Alternating and symmetric groups}

In this section we consider the minimal base sizes of alternating and symmetric groups in primitive actions. Here is the main result.

\begin{thm}
\label{generalalternating}
Let $G$ be a primitive permutation group of degree $n$ with socle isomorphic to $\mathrm{Alt}(m)$ for some integer $m \geq 5$. Then   
$$b(G) \leq 2  \frac{\log |G|}{\log n} + 16.$$
\end{thm}

In the proof of Theorem \ref{generalalternating}, we may assume 
that $19 \leq b(G) \leq \log|G|$. In particular $m \geq 7$ and $G = \mathrm{Alt}(m)$ or $\mathrm{Sym}(m)$.

Let $\O$ be a set of size $n$ permuted by $G$, let $\a \in \O$ and let $H = G_\a$, a maximal subgroup of $G$. There are three possiblities to consider, according to the action of $H$ on the underlying set $\{1,\ldots,m\}$:
\begin{itemize}
\item[(1)] $H$ is intransitive: here $H = (\mathrm{Sym}(k) \times \mathrm{Sym}(m-k))\cap G$ for some $k\le m/2$;
\item[(2)] $H$ is transitive and imprimitive: here $H = (\mathrm{Sym}(b) \wr \mathrm{Sym}(a))\cap G$, where $m=ab$;
\item[(3)] $H$ is primitive on $\{1,\ldots,m\}$.
\end{itemize}
In case (1), the action of $G$ on $\O$ is the action on $k$-element subsets of $\{1,\ldots,m\}$, and in case (2) the action is on partitions into $a$ parts of size $b$. These actions are considered in Sections \ref{Section2.1} and \ref{Section2.2}, and the proof of Theorem \ref{generalalternating} is completed in Section \ref{Section2.3}.

\subsection{Action on subsets}
\label{Section2.1}

Here we prove Theorem \ref{generalalternating} in the case when the action is on subsets (see Proposition \ref{maink}). 
Let $\mathrm{Sym}(m)$ act on the set $\Omega(m,k)$ of all $k$-element subsets of the set $\{ 1, \ldots , m \}$, where $k\le m/2$. Set $n = |\Omega(m,k)| = \binom{m}{k}$. Let $b(m,k)$ denote the minimal size of a base for 
$\mathrm{Sym}(m)$ acting on $\Omega(m,k)$. For convenience set $t = m/k$.

A detailed study of the function $b(m,k)$ was carried out in \cite{H12}. Here are the main results from that paper that we need.

\begin{thm} {\rm (\cite[Thm. 3.2, Cor. 4.3]{H12})}
\label{precise}
\begin{itemize}
\item[{\rm (i)}] We have $b(m,k)\leq \left\lceil\log_{\lceil t\rceil}(m)\right\rceil
\left(\lceil t\rceil-1\right).$

\item[{\rm (ii)}] If $k^{2} \leq m$, then 
$$b(m,k) = \Big\lceil \frac{2m-2}{k+1} \Big\rceil < \frac{2m}{k+1} + 1 = \frac{2k}{k+1} t + 1.$$ 
\end{itemize}
\end{thm}

We shall need the following estimates for $\ln |\mathrm{Sym}(m)| / \ln |\Omega(m,k)|$. 

\begin{lem}
\label{binom}
We have $$\Big( \frac{t}{\ln(t)+1} \Big)   (\ln m - 1) < \frac{\ln |\mathrm{Sym}(m)|}{\ln |\Omega(m,k)|} < \Big( \frac{t}{\ln(t)} \Big) \ln m.$$
\end{lem}

\begin{proof}
By the inequalities 
$${(m/k)}^{k} < \binom{m}{k} < {(me/k)}^{k} \hbox{ and } {(m/e)}^{m} < m! < m^{m},$$
we have 
$$\frac{m (\ln m - 1)}{k (\ln (m/k) + 1)} < \frac{\ln |\mathrm{Sym}(m)|}{\ln |\Omega(m,k)|} < \frac{m \ln m}{k \ln (m/k)} = \frac{m/k}{\ln(m/k)} \ln m.$$ 
From this the lemma follows. 
\end{proof}

The next result establishes the conclusion of Theorem \ref{generalalternating}  under the assumption that $k^{2} \leq m$.  

\begin{lem}
\label{klarge}
Assume that $k^{2} \leq m$. Then 
$$b(m,k) < 2 \frac{\ln |\mathrm{Sym}(m)|}{\ln |\Omega(m,k)|} + 4.$$
\end{lem}

\begin{proof}
Assume first that $k \geq 8 > e^{2}$. By Theorem \ref{precise}(ii) and Lemma \ref{binom}, we have 
$$\frac{b(m,k) \ln n}{\ln |\mathrm{Sym}(m)|} < \Big( \frac{2t+1}{t} \Big) \Big( \frac{\ln(t)+1}{\ln(m)-1} \Big).$$ 
Since $k \geq 8 > e^{2}$, it follows that $\frac{\ln(t)+1}{\ln(m)-1} < 1$. By this and Lemma \ref{binom},  
$$b(m,k) < 2 \frac{\ln |\mathrm{Sym}(m)|}{\ln n} + \Big( \frac{\ln (m)}{\ln(m)-1} \Big) \Big( \frac{\ln(t) +1}{\ln(t)} \Big).$$ It is easy to see that the second term is less than $4$, giving the conclusion in this case ($k \geq 8 >e^2$). 

Hence we may assume that $k \leq 7$. A GAP \cite{GAP} computation shows that the bound in the conclusion of the lemma holds for $5 \leq m \leq 148 < e^{5}$. Thus assume also that $m \geq 149 > e^{5}$. 


If $2 \leq k \leq 7$ then Theorem \ref{precise} gives $b(m,k) < \frac{2k}{k+1}t + 1$, and so by Lemma \ref{binom},
$$\frac{b(m,k) \ln n}{\ln |\mathrm{Sym}(m)|} < \Big(\frac{2k}{k+1} + \frac{k}{m}\Big) \Big(\frac{\ln m - \ln k + 1}{\ln m -1}\Big).$$
This is less than $2$ for $m \geq 149 > e^{5}$.
\end{proof}

Here is the main result of this subsection.

\begin{prop} \label{maink}
We have $$b(m,k) \leq 2 \frac{\ln |\mathrm{Sym}(m)|}{\ln |\Omega(m,k)|} + 16.$$
\end{prop}

\begin{proof}
By Lemma \ref{klarge}, we may assume that $k^{2} > m$, which is equivalent to saying that $t^{2} < m$. 

Define $r$ to be the integer $r \geq 2$ with $t^{r} < m \leq t^{r+1}$. Then by Theorem \ref{precise}(i), we have 
$b(m,k) \leq (r+1)t$. By Lemma \ref{binom}, this gives
$$\frac{b(m,k) \ln n}{\ln m!} < \frac{(r+1)(\ln t + 1)}{r \ln t - 1}.$$ 
A GAP \cite{GAP} computation shows that the right hand side is less than $2$ provided that $r = 2$ and $t \geq 149 > e^{5}$, or $r = 3$ and $t \geq 20 > e^{3}$, or $r \geq 4$ and $t \geq 11$. 

If $r = 3$ and $t \leq 20 < e^3$, then $4t - 2 (\frac{3 \ln t - 1}{\ln t + 1} ) t \leq 16$, which gives the conclusion (using Lemma \ref{binom}). Similarly, if $r \geq 4$ and $t < 11$, then $(r+1)t - 2 (\frac{r \ln t - 1}{\ln t + 1} ) t \leq 11$, giving the conclusion.

This leaves the case where $r=2$ and $t \leq 148 < e^5$. We first distinguish eleven different cases according to some possible ranges of values of $\ln t$. If $\ln t$ falls in any of the intervals $[\epsilon,\epsilon + 0.2]$ where $\epsilon = 2.8 + 0.2 \ell$ and $\ell$ is a non-negative integer at most $10$, then $3t - 2 (\frac{2 \ln t - 1}{\ln t + 1} ) t \leq 16$. Thus we may assume that $t < e^{2.8}$. But then $m \leq t^{3} < e^{8.4} < 4500$. By a GAP \cite{GAP} calculation, we see that if $5 \leq m \leq 4500$, then $3t - 2 (\frac{2 \ln t - 1}{\ln t + 1} ) t \leq 11$. This completes the proof. 
\end{proof}

The final result of this subsection gives the first part of Proposition \ref{acc}.

\begin{prop}
\label{2k/(k+1)}
Fix a positive integer $k$. Then as $m \to \infty$,
$$\frac{b(m,k) \log|\Omega(m,k)|}{\log|\mathrm{Sym}(m)|} \to 2k/(k+1).$$ 
\end{prop}

\begin{proof}
Assume that $m \geq k^{2}$. Then, by Theorem \ref{precise}(ii), $b(m,k)=\left\lceil\frac{2m-2}{k+1}\right\rceil$, and hence $b(m,k)/m \to \frac{2}{k+1}$ as $m\to \infty$. Also 
$(m \ln|\Omega(m,k)| / \ln|\mathrm{Sym}(m)|) \to k$ by Lemma \ref{binom}. The result follows.  
\end{proof}

\subsection{Action on partitions}
\label{Section2.2}

Now consider the minimal base size $f(a,b)$ of the group $\mathrm{Sym}(m)$ acting on the set $\Omega$ of all partitions of $\{ 1, \ldots, m \}$ into $a$ parts each of size $b$, where $m = ab$ and $a$, $b \geq 2$. In this case $n = |\Omega| = m!/({b!}^{a}a!)$. Bases for this action were studied in \cite{BCN}, where the following was proved.

\begin{thm}\label{bcnprop} {\rm (\cite{BCN})} Suppose $b\ge 3$. Then one of the following holds:
\begin{itemize}
\item[{\rm (i)}] $a\ge b$ and $f(a,b)\le 6$;
\item[{\rm (ii)}] $a<b$ and $f(a,b) \leq \log_{a}(b) + 4$.
\end{itemize}
\end{thm}

We shall need the following bound.

\begin{lem}
Let $a,b$ be integers with $2\le a<b$. Then 
$$\frac{\ln b}{\ln a} - 1 < \frac{\ln((ab)!)}{\ln \Big(\frac{(ab)!}{{(b!)}^{a}a!}\Big)}.$$
\end{lem} 

\begin{proof}
Write $g(a,b) = \ln((ab)!)\,/\,\ln \Big(\frac{(ab)!}{{(b!)}^{a}a!}\Big)$. Then 
using the bounds 
$$\sqrt{2 \pi} \cdot {\ell}^{1/2} {\Big(\frac{\ell}{e}\Big)}^{\ell} < \ell ! < e \cdot {\ell}^{1/2} {\Big(\frac{\ell}{e}\Big)}^{\ell}$$ 
which hold for all positive integers $\ell$, we have 
\[
\begin{array}{rl}
g(a,b) >& \frac{ab (\ln (ab)-1)}{\ln((ab)!) - a \ln(b!) - \ln(a!)} \\
>& \frac{ab(\ln(ab)-1)}{ \ln(e/\sqrt{2\pi}) + ab(\ln a) + \frac{1}{2}\ln(ab) - a \ln(\sqrt{2\pi}) - \frac{1}{2}a \ln b - \frac{1}{2}\ln a - a \ln a + a}\\
 =& \frac{\ln(ab)-1}{\ln a + \frac{1}{b}(1 - \ln a - \frac{1}{2}\ln(2 \pi)) + \ln(e/\sqrt{2\pi})/(ab) + \frac{1}{2b} ( (\ln (b))/a - \ln b ) }\\
 \geq & \frac{\ln(ab)-1}{\ln a + \frac{1}{b}(1 - \ln a - \frac{1}{2}\ln(2 \pi) - \frac{\ln b}{4}) + \ln(e/\sqrt{2\pi})/(ab) } \geq \\
\geq & \frac{\ln(ab)-1}{\ln a + \frac{1}{b}(1 - \ln 2 - \frac{1}{2}\ln(2 \pi) - \frac{\ln 3}{4} + \ln(e/\sqrt{2\pi})/2) } \\
 > & \frac{\ln(ab)-1}{\ln a - (0.8)/b} > \frac{\ln b}{\ln a} - 1.
\end{array}
\]
\end{proof}

Here is the main result of this subsection.

\begin{prop}\label{partition}
With the above notation, we have $f(a,b) \leq \frac{\ln |\mathrm{Sym}(m)|}{\ln n} + 5$.
\end{prop}

\begin{proof}
If $b\ge 3$, this follows immediately from Theorem \ref{bcnprop}. And for $b=2$, 
Remark 1.6(ii) of \cite{BGS} gives $f(a,2) \leq 3$. 
\end{proof}

\subsection{Proof of Theorem \ref{generalalternating}} \label{Section2.3}

Let $G = \mathrm{Alt}(m)$ or $\mathrm{Sym}(m)$ act primitively on a set $\O$, and let $H$ be a point-stabilizer in $G$. The cases where $\O$ is a set of $k$-subsets or partitions of $\{1,\ldots ,m\}$ have been dealt with in Propositions \ref{maink} and \ref{partition}. Hence by the remarks at the beginning of the section, we may assume that $H$ is primitive on $\{1,\ldots ,m\}$. In this case, it is proved in \cite[Cor. 2]{BGS} that $b_{\Omega}(G) \leq 5$. This completes the proof of Theorem \ref{generalalternating}.

\section{Classical groups}
\label{Section3}

In this section we study base sizes of primitive actions of classical groups. Our main result is the following. 

\begin{thm}
\label{generalclassical}
Let $G$ be an almost simple primitive permutation group of degree $n$ whose socle is a classical simple group. 
Then $b(G) \leq 2 (\log|G| / \log n) + 16$.
\end{thm}      

We shall divide the proof of this theorem into several subcases. First we give a definition, taken from \cite{LS99}. Let $G$ be an almost simple group with socle $G_0$, a classical group with natural module $V$, a vector space of dimension $d$  over a field $\F_q$ of characteristic $p$. We call a maximal subgroup $M$ of $G$ a {\it subspace subgroup} if it is reducible on $V$, or is an orthogonal group on $V$ embedded in a symplectic group with $p=2$; more specifically, $M$ is a subspace subgroup if one of the following holds:
\begin{itemize}
\item[(1)] $M=G_U$ for some proper nonzero subspace $U$ of $V$, where $U$ is totally singular, non-degenerate, or, if $G$ is orthogonal and $p=2$, a nonsingular 1-space ($U$ is any subspace if $G_0=PSL(V)$);
\item[(2)] $G_0 = PSL(V)$, $G$ contains a graph automorphism of $G_0$, and $M\cap G_0 = (G_0)_{U,W}$ where $U,W$ are proper nonzero subspaces of $V$, $\dim V = \dim U+\dim W$ and either $U\subseteq W$ or $V = U\oplus W$;
\item[(3)] $G_0 = Sp_{2m}(q)$, $p=2$ and $M\cap G_0 = O^{\pm}_{2m}(q)$.
\end{itemize}
Note that in (3), if we regard $G_0$ as the isomorphic orthogonal group $O_{2m+1}(q)$, then $M\cap G_0 = O_{2m}^\pm(q)$ is the stabilizer of a hyperplane of the natural module of dimension $2m+1$.

If $M$ is a subspace subgroup, we call the action of $G$ on the coset space $G/M$ a {\it subspace action}.

Bases for non-subspace actions of classical groups were studied in detail in \cite{Burness}, so our main task is to prove Theorem \ref{generalclassical} for subspace actions. First we require the following general bound.

\begin{prop} \label{prop:ClassicalBound_nperk}
Let $G$ be as above, and suppose $M$ is as in $(1)$, so that the coset space $X = G/M$ is a $G$-orbit of $k$-dimensional subspaces of $V$, for some $k$. Assume also that $k \le d/2$. Then 
\[
 \frac{\log|G|}{\log|X|} \ge \frac{d}{tk}-1,
\]
where $t=1$ if $G_0 = PSL(V)$, and $t=2$ otherwise.
\end{prop}

\begin{proof}
Observe that $|G| > q^{(d^2/t)-d}$, while
  \[
  |X|\le \binom{d}{k}_q:=\frac{(q^d-1)(q^d-q)\cdots (q^d-q^{k-1})}
  {(q^k-1)(q^k-q)\cdots (q^k-q^{k-1})}\leq \left(\frac{q^d}{q^{k-1}}\right)^k=
  q^{dk-k^2+k}.
  \]
Hence 
  \[\frac{\log|G|}{\log|X|}\geq
  \frac{(d^2/t)-d}{dk-k^2+k}\geq \frac{d}{tk}-1.
  \]
\end{proof}

\subsection{Action on an orbit of subspaces}
\label{classicalorbit}

In this subsection we prove Theorem \ref{generalclassical} for subspace actions of classical simple groups as in case (1) in the list above. This is the main part of the proof of the theorem.

\begin{thm}
\label{classicalmain}
  Let $G$ be a simple classical group on $V$, a vector space of dimension $d$ over $\F_q$. Let $X$ 
  be a $G$-orbit of $k$-dimensional subspaces of $V$ with $k \leq d/2$, 
  on which $G$ acts primitively. Then 
  $$b_X(G)\leq \frac{d}k+11.$$ 
\end{thm}

\begin{proof}
  In every subcase, we will define a base $\mB\subset X$ of $G$ in a
  number of steps.  We do this by starting with $\mB=\emptyset$ and at
  each step adding some subspaces to $\mB$. Throughout the proof, $G_{(\mB)}$ denotes the pointwise stabilizer of $\mB$ -- that is, the set of group elements that fix all the subspaces in $\mB$.
  
\vspace{2mm}
\noindent  \emph{Action on the set of all $k$-dimensional subspaces.}

  First, let us assume that $X$ is the set of all
  $k$-dimensional subspaces. Let $d=ak+r$ for $a\geq 2$ and $0\leq r<k$.
  Take any direct sum decomposition $V=V_1\oplus \ldots\oplus V_a\oplus U$ with 
  $\dim V_i=k$ for $1\leq i\leq a$, and let $V_1,\ldots,V_a\in \mB$. 
  Fix a basis $B_i=\{x_1^{(i)},\ldots, x_k^{(i)}\}\in V_i$ for each $i$ and 
  define $W_1=\langle \sum_{i=1}^a x_s^{(i)}\,|\,1\leq s\leq k\rangle\in X$
  and put $W_1$ into $\mB$. Then the matrix form of 
  the restriction of a $g\in G_{(\mB)}$ to $V_1\oplus\ldots\oplus V_a$ is 
  a block diagonal matrix with equal blocks (with respect to
  the basis $B_1\cup\ldots\cup B_a$). 
  Define $M_g=[g_{V_1}]_{B_1}=[g_{V_2}]_{B_2}=\ldots=[g_{V_a}]_{B_a}$ and
  let $\{\gamma,\delta\}\subset SL(V_2)$ be a generating set of $SL(V_2)$ 
  and $C=[\gamma]_{B_2},\,D=[\delta]_{B_2}$ be their matrix forms. 
  Then $g\in G_{(\mB)}$ fixes the subspaces 
  \[
  W_2=\langle x_s^{(1)}+\gamma(x_s^{(2)})\,|\,1\leq s\leq k\rangle\in X,\quad
  W_3=\langle x_s^{(1)}+\delta(x_s^{(2)})\,|\,1\leq s\leq k\rangle\in X
  \]
  if and only if $M_g$ commutes with both $C$ and $D$. Thus, 
  putting $W_2$ and $W_3$ into $\mB$, it follows that $G_{(\mB)}$ acts as scalars 
  on $V_1\oplus\ldots\oplus V_a$. 
  Finally, if $r>0$ then let 
  $\{f_{1},\ldots,f_r\}$ be a basis of $U$ and define
  \[
  W_4=\langle f_1,\ldots,f_r,x_{r+1}^{(1)},\ldots,x_{k}^{(1)}\rangle,\qquad
  W_5=\langle f_1+x_1^{(2)},\ldots,f_r+x_r^{(2)},x_{r+1}^{(2)},\ldots,x_{k}^{(2)}
  \rangle.
  \]
  Adding $W_4$ and $W_5$ to $\mB$ it is easy to see that $G_{(\mB)}$
  contains only scalar transformations. Thus, $b_X(G)\leq a+5\leq
  \frac{d}k+5$ for this case.

\vspace{2mm}
\noindent   \emph{Action on an orbit of non-degenerate subspaces.} 

Now we turn
  to the case when $G$ is a group fixing some non-degenerate form $[\,,\,]$ on
  $V$ and $X$ is a $G$-orbit of non-degenerate subspaces. In the special 
  case $d=2k$, we also assume that the Witt index of elements of $X$ is
  no more than the Witt index of elements of $X^\perp$ (this is 
  only interesting in the orthogonal case, when $k$ is even and 
  $V$ has Witt index $k-1$). This will guarantee that the sums defining 
  $u_i$ and $v_i$ below  will have at least two terms.

  Let $d=ak+r$ with $1\leq r\leq k$ and take any
  orthogonal decomposition $V=V_1\oplus \ldots\oplus V_a\oplus
  V_{a+1}$ with $V_1,\ldots,V_a\in X$. Put $V_1,\ldots,V_a$
  into $\mB$. Then any $g\in G_{(\mB)}$ also fixes $V_{a+1}=(\sum_{i=1}^a
  V_i)^\perp$. 

  Let $l$ and $m\leq 2$ denote the Witt index and the Witt defect of
  the subspaces in $X$, so $k=2l+m$.  First, let us assume that $l\geq
  1$.  Then for every $1\leq s\leq a$, the subspace $V_s$ is a direct
  sum of orthogonal subspaces $V_s=V_s^{(h)}\oplus V_s^{(m)}$, where
  each $V_s^{(h)}$ contains a basis
  $B_s=\{x_1^{(s)},\ldots,x_l^{(s)},y_1^{(s)},\ldots,y_l^{(s)}\}$ with
  $[x_i^{(s)},x_j^{(s)}]=[y_i^{(s)},y_j^{(s)}]=0,\,[x_i^{(s)},y_j^{(s)}]=\delta_{ij}$
  for all $i,j$ and $V_s^{(m)}$ has dimension and Witt defect
  $m$. (For orthogonal groups of characteristic 2, we also have
  $Q(x_i^{(s)})=Q(y_i^{(s)})=0$ for all $i$, where $Q$ is the underlying quadratic form.)  Furthermore, let $l'$
  be the minimum of $l$ and the Witt index of $V_{a+1}$ and $m'=\dim
  (V_{a+1})-2l'$. Similarly to the above, choose an orthogonal
  decomposition $V_{a+1}=V_{a+1}^{(h)}\oplus V_{a+1}^{(m)}$ along with
  a basis $B_{a+1}=\{x_1^{(a+1)},\ldots,x_{l'}^{(a+1)},
  y_1^{(a+1)},\ldots,y_{l'}^{(a+1)}\}$ of $V_{a+1}^{(h)}$ satisfying
  $[x_i^{(a+1)},x_j^{(a+1)}]=[y_i^{(a+1)},y_j^{(a+1)}]=0,\,
  [x_i^{(a+1)},y_j^{(a+1)}]=\delta_{ij}$ for $1\leq i,j\leq l'$,
  and define $x_i^{(a+1)}=y_i^{(a+1)}=0$ for $l'<i\leq l$. 
  
For $1\leq i\leq l$, define
\[
u_i=\sum_{s=1}^{a+1}  x_i^{(s)},\, v_i=\sum_{s=1}^{(a+1)} y_i^{(s)}.
\]
We define the subspaces
  \begin{align*}
  W_1^{(h)}&=\langle u_1,\ldots,u_l,y_1^{(1)},\ldots,y_l^{(1)}\rangle,\quad
  &W_2^{(h)}&=\langle x_1^{(1)},\ldots,x_l^{(1)},v_1,\ldots,v_l\rangle,\\
  W_3^{(h)}&=\langle u_1,\ldots,u_l,y_1^{(2)},\ldots,y_l^{(2)}\rangle,\quad
  &W_4^{(h)}&=\langle x_1^{(2)},\ldots,x_l^{(2)},v_1,\ldots,v_l\rangle.
  \end{align*}
  Then $W_t:=W_t^{(h)}\oplus V_1^{(m)}\in X$ for each $1\leq t\leq 4$. 
  Adding $W_1,W_2,W_3,W_4$ to $\mB$, we see that any $g\in G_{(\mB)}$ fixes 
  each $V_s^{(h)}$ and, moreover, the matrix form of each restriction
  $g_{V_s^{(h)}}$ satisfies
  \[
  \left[g_{V_1^{(h)}}\right]_{B_1}=
  \left[g_{V_2^{(h)}}\right]_{B_2}=\ldots=
  \left[g_{V_a^{(h)}}\right]_{B_a}=\begin{pmatrix}A_g&0\\0&B_g\end{pmatrix}
  \]
  for some $A_g,B_g\in GL(l,q)$.
  The use of the $x_i^{(a+1)},y_i^{(a+1)}$ as summands in the $u_i$ and $v_i$ 
  also implies that 
\[
\left[g_{V_{a+1}^{(h)}}\right]_{B_{a+1}}=
  \begin{pmatrix}A_g'&0\\0&B_g'\end{pmatrix},
\]
 where $A_g'$ and $B_g'$ are left  upper $l'\times l'$ submatrices of $A_g$ and $B_g$, respectively. 

  Adding also the subspace $W_5=W_5^{(h)}\oplus V_1^{(m)}$ with
  $W_5^{(h)}:=\langle x_i^{(1)},y_i^{(1)}+x_i^{(2)}\,|\,1\leq i\leq
  l\rangle$ we can also guarantee that $A_g=B_g$ holds for any $g\in
  G_{(\mB)}$.

  Let $B_2^{(x)}=\{ x_1^{(2)},\ldots,x_l^{(2)}\}$ and 
  $V_2^{(x)}$ be the subspace spanned by $B_2^{(x)}$. Choose
  $\varphi,\psi\in \textrm{End}(V_2^{(x)})$ generating
  $\textrm{End}(V_2^{(x)})$ as an algebra.  Define
  \[
  W_6^{(h)}=\langle x_i^{(1)}+\varphi(x_i^{(2)}),
  y_i^{(1)}\,|\,1\leq i\leq l\rangle,\quad
  W_7^{(h)}=\langle x_i^{(1)}+\psi(x_i^{(2)}),
  y_i^{(1)}\,|\,1\leq i\leq l\rangle.
  \]
  and let $W_6:=W_6^{(h)}\oplus V_1^{(m)},\ W_7:=W_7^{(h)}\oplus V_1^{(m)}$.
  Adding $W_6,W_7$ to $\mB$, we see that $g_{V_2^{(x)}}$ commutes with 
  both $\varphi$ and $\psi$ for any $g\in G_{(\mB)}$. 
  Thus $g_{V_2^{(x)}}$ is a scalar transformation. Therefore, any
  $g\in G_{(\mB)}$ acts as a scalar on $V_1^{(h)}\oplus\ldots\oplus V_{a+1}^{(h)}$. 
 
  For every $1\leq s\leq a+1$, 
  choose other orthogonal decompositions $V_s=V_s^{(h)'}\oplus
  V_s^{(m)'}$ such that each $V_s^{(h)'}$ is isometric to $V_s^{(h)}$ 
  and $V_s^{(h)}+V_s^{(h)'}=V_s$.  Applying similar
  constructions as for $W_1,\ldots,W_4$ before, by adding $4$ further 
  subspaces to $\mB$ we can synchronize the 
  action of any $g\in G_{(\mB)}$ on each $V_s^{(h)'}$ with its
  action on each $V_t^{(h)}$. Thus, now each $g\in G_{(\mB)}$ acts as a scalar on the 
  whole vector space $V_1\oplus\ldots\oplus V_{a+1}$. Thus, 
  $b_X(G)\leq a+7+4\leq \frac{d}{k}+11$ in this case.
  
  Now, let us assume that $l=0$, so $k=m\leq 2$. The case $k=1$ is trivial.
  The case $k=m=2$ implies that $V$ is orthogonal. 
  We can also assume that $a\geq 3$, since otherwise $d\leq 6$.
  Then each $V_s$ has a basis $x^{(s)},y^{(s)}$ with 
  $Q(x^{(s)})=1,\, Q(y^{(s)})=\alpha,\,[x^{(s)},y^{(s)}]=1$, 
  where $\alpha\in\FF q$ is such that the polynomial 
  $t^2+t+\alpha$ is irreducible over $\FF q$. Additionally, choose an
  arbitrary spanning set $\{ x^{a+1},y^{a+1} \}$ of $V_{a+1}$. 
  Since $\{Q(z)\,|\,z\in V_s\}=\FF q$, 
  we can define 
  \[u_1=\sum_{s=2}^{a+1} x^{(s)}+z^{(1)},\quad 
    v_1=\sum_{s=2}^{a+1} y^{(s)}+w^{(1)},\quad
    u_2=\sum_{s=1}^{a-1} x^{(s)}+z^{(a)},\quad 
    v_2=\sum_{s=1}^{a-1}y^{(s)}+w^{(a)}
  \]
  with $z^{(1)},w^{(1)}\in V_1,\,z^{(a)},w^{(a)}\in V_a$ such that 
  $Q(u_1)=Q(u_2)=1,\,Q(v_1)=Q(v_2)=\alpha$.
  Now, let
  \begin{gather*}
  W_1=\langle u_1   ,y^{(2)}\rangle,\quad
  W_2=\langle u_1   ,y^{(3)}\rangle,\quad
  W_3=\langle x^{(2)},v_1   \rangle,\quad
  W_4=\langle x^{(3)},v_1   \rangle,\\
  W_5=\langle u_2   ,y^{(1)}\rangle,\quad
  W_6=\langle u_2   ,y^{(2)}\rangle,\quad
  W_7=\langle x^{(1)},v_2   \rangle,\quad
  W_8=\langle x^{(2)},v_2   \rangle.
  \end{gather*}
  Adding each $W_i$ to $\mB$ we see that the restriction of 
  any $g\in G_{(\mB)}$ to the subspaces $V_1,\ldots,V_a$ has matrix form
  \[
  \Big[g_{V_1}\Big]_{\{x^{(1)},y^{(1)}\}}=\ldots=
  \Big[g_{V_a}\Big]_{\{x^{(a)},y^{(a)}\}}=
  \begin{pmatrix}c&0\\0&d\end{pmatrix}
  \textrm{ for some }c,d\in \FF q.
  \]
  Using that $Q(g(x^s))=1,\,Q(g(y^{(s)}))=\alpha,\,[g(x^{(s)},g(y^{(s)})]=1$ 
  it follows that $c=m=\pm 1$, so any $g\in G_{(\mB)}$ acts on 
  $V_1\oplus\ldots\oplus V_a$ as a scalar transformation. 
  The use of $x^{(a+1)},y^{(a+1)}$ guarantees that $g\in G_{(\mB)}$ is a scalar on the 
  whole of $V$. Thus, $b_X(G)\leq \frac{d}k+8$ for this case. 

\vspace{2mm}
\noindent  \emph{Action on an orbit of totally singular subspaces.} 

From now on, let $X$ be the set of $k$-dimensional totally singular subspaces
  of $V$. Again, we can assume that $k\geq 2$. 

  Let $l$ be the Witt index and let $m\leq 2$ be the Witt defect of $V$, so
  $k\leq l$ (since otherwise $X=\emptyset$) and $d =2l+m$.  Let
  $l=ak+r$ for $0\leq r<k$ and denote $w(s)=k$ for $1\leq s\leq k$ and 
  $w(a+1)=r$. 
  Take an orthogonal decomposition
  $V=V_1\oplus\ldots\oplus V_a\oplus V_{a+1}\oplus U$ such that $V_s$ has
  dimension $2w(s)$ and Witt index $w(s)$ for each $1\leq s\leq a+1$.  For
  every $1\leq s\leq a+1$ let
  $B_s=\{x_1^{(s)},\ldots,x_{w(s)}^{(s)},y_1^{(s)},\ldots,y_{w(s)}^{(s)}\}$ be a
  basis of $V_s$ such that 
  $V_s^{(x)}=\langle x_1^{(s)},\ldots,x_{w(s)}^{(s)}\rangle,
  \ V_s^{(y)}=\langle y_1^{(s)},\ldots,y_{w(s)}^{(s)}\rangle$ are 
  $w(s)$-dimensional singular subspaces, and 
  $[x_i^{(s)},y_j^{(s)}]=\delta_{ij}$ for every $1\leq i,j\leq w(s)$. 
  Furthermore, define $x_i^{(a+1)}=y_i^{(a+1)}=0$ for $r<i\leq k$. 
  Finally, take the additional 
  $k$-dimensional singular subspaces 
  \[
\begin{array}{l}
  V_{a+1}^{(x)'}=\langle x_1^{(a+1)},\ldots,x_r^{(a+1)},
  x_{r+1}^{(1)},\ldots,x_k^{(1)}\rangle,\\
  V_{a+1}^{(y)'}=\langle y_1^{(a+1)},\ldots,y_r^{(a+1)},
  y_{r+1}^{(1)},\ldots,y_k^{(1)}\rangle.
\end{array}
  \]
Let $u_i=\sum_{s=1}^{(a+1)} x_i^{(s)},\,
  v_i=\sum_{s=1}^{(a+1)} y_i^{(s)}$ for $1\leq i\leq k$ and define
  \[
  W_1=\langle u_1,\ldots,u_k\rangle,\qquad 
  W_2=\langle v_1,\ldots,v_k\rangle.
  \]
  First, add each of the subspaces $V_1^{(x)},V_1^{(y)},\ldots,V_a^{(x)},V_a^{(y)},
  V_{a+1}^{(x)'},V_{a+1}^{(y)'},W_1,W_2$ to $\mB$. 
  Then the subspaces $V_s$ for each $1\leq s\leq a+1$ are fixed by any 
  $g\in G_{(\mB)}$
  and the restrictions $g_{V_s}$ have matrix form 
  \begin{gather*}
  \Big[g_{V_1}\Big]_{B_1}=\ldots=\Big[g_{V_a}\Big]_{B_a}=
  \begin{pmatrix}A_{g}&0\\0&(A_{g})^{-T}\end{pmatrix},\\
  \Big[g_{V_{a+1}}\Big]_{B_{a+1}}=
  \begin{pmatrix}A'_{g}&0\\0&(A'_{g})^{-T}\end{pmatrix},
  \end{gather*}
  where $A_g\in GL(k,q)$ and $A'_g$ is the left upper $r\times r$ submatrix of 
  $A_g$. 

  Next, we define additional $k$-dimensional singular subspaces 
  of the form
  \begin{align*}
  W^{(x)}(C)&=\Big\langle \sum_{s=1}^{a+1}\Big(x_j^{(s)}+
    \sum_{i=1}^k c_{ij}y_i^{(s)}\Big)\,\Big|
    \,1\leq j\leq k\Big\rangle,\\
  W^{(y)}(C)&=\Big\langle \sum_{s=1}^{a+1}\Big(y_j^{(s)}+
    \sum_{i=1}^k c_{ij}x_j^{(s)}\Big)\,\Big|
    \,1\leq j\leq k\Big\rangle,
  \end{align*}
  where $C=(c_{ij})\in M(k,q)$.  The subspaces $W^{(x)}(C)$ and $W^{(y)}(C)$
  are singular if the matrix $C$ is symmetric (when $V$ is a
  symplectic space) or anti-symmetric (when $V$ is an orthogonal or a
  unitary space).  Furthermore, $g\in G_{(\mB)}$ fixes $W^{(x)}(C)$
  (resp. $W^{(y)}(C)$) if and only if $A_g^T C=CA_g^{-1}$
  (resp. $A_g C=CA_g^{-T})$ holds.
  
  First, let us assume that $V$ is a symplectic space and choose
  $C,D\in M(k,q)$ symmetric matrices, which generate the full matrix
  algebra $M(k,q)$ (as an algebra). Adding $W^{(y)}(I), W^{(y)}(C),
  W^{(y)}(D)$ to $\mB$ we see that any $g\in G_{(\mB)}$ satisfies
  $A_{g}=A_g^{-T}$, and, therefore, $A_{g}C=CA_g,\,
  A_{g}D=DA_g$. It follows that $A_g$
  is a scalar matrix for any $g\in G_{(\mB)}$. Thus, $g$ acts as a scalar on the 
  whole $V=V_1\oplus\ldots\oplus V_a\oplus V_{a+1}$. 

  Now, let $V$ be an orthogonal or unitary space and 
  choose antisymmetric matrices $C=E_{12}-E_{21},\,
  D=\sum_{i=2}^{k-1}(E_{i,i+1}-E_{i+1,i})$. (Here 
  $\{E_{ij}\,|\,1\leq i,j\leq k\}$ denotes the usual basis of the full matrix 
  algebra $M(k,q)$.) 
  Add the subspaces $W^{(x)}(C),W^{(y)}(C),W^{(x)}(D),W^{(y)}(D)$ to $\mB$ and 
  let $g\in G_{(\mB)}$. Then we have 
  $A_g C A_g^T=A_g^T CA_g=C$ and $A_g D A_g^T=A_g^T DA_g=D$.
  Using the implication 
  \[
AXA^T=X,\,A^TYA=Y \Rightarrow  A XY= A XA^TYA=XYA,
  \]
  we see that $A_g$ commutes with every product $P$ of $C$'s and $D$'s with an
  even number of terms. Similarly, $A_g PA_g^T=A_g^TPA_g=P$ holds for
  every product $P$ of $C$'s and $D$'s with an odd number of terms.  In
  particular, $A_g$ commutes with
  $CD=E_{13},\,-DC=E_{31},-CD^2C=E_{11}$ etc.  Continuing this way, we
  see that $A_g=\lambda\cdot I$ is a scalar matrix.  The equation $A_g
  C A_g^T=C$ also shows that $\lambda^2=1$ and so $A_g=A_g^{-T}=\pm
  I$. Thus, any $g\in G_{(\mB)}$ is a scalar transformation on
  $V_1\oplus\ldots\oplus V_{a+1}$. If $U\neq 0$, we can choose
  a $k$-dimensional singular subspace $V_{a+2}^{(x)}\leq V_1\oplus U$
  satisfying $V_1+V_{a+2}^{(x)}=V_1+U$. Let
  $x_1^{(a+2)},\ldots,x_k^{(a+2)}$ be any basis of
  $V_{a+2}^{(x)}$. Adding the subspaces $V_{a+2}^{(x)},\ \langle
  x_1^{(a+2)}+y_1^{(1)},\ldots,x_k^{(a+2)}+y_k^{(1)}\rangle$ to $\mB$
  gives the result. So, $b_X(G)\leq 2a+10\leq \frac{d}k+10$.

  The above argument works if $X$ is the set of all totally singular
  subspaces, which is indeed a $G$-orbit in most cases.  The only
  exception is when $V$ is an orthogonal space, $d=2k$, and
  $G = \O^+(V)$, so we assume this from now on.  Then
  two totally singular $k$-dimensional subspaces $V_1,V_2$ are in the same $G$-orbit if
  and only if $\dim (V_1\cap V_2)\equiv k\pmod 2$. Since the full 
  orthogonal group $O(V)$ interchanges the two  $G$-orbits, it
  does not matter, which orbit we choose.  Note that in the above
  construction the subspaces $V_1^{(x)}, V_1^{(y)},
  W^{(x)}(C),W^{(y)}(C),W^{(x)}(D),W^{(y)}(D)$ are in the same
  $G$-orbit provided that $k$ is even (the further subspaces defined
  in the proof are now meaningless). So, $b_X(G)\leq 6$ in this
  case.  Now, let $k$ be odd, and choose an orthogonal decomposition
  $V=\langle x,y\rangle \oplus U$, where $x,y$ is a hyperbolic pair.
  Then $\dim U=d-2=2(k-1)$, so the above construction works for a
  $G_U$-orbit of ($k-1$)-dimensional totally singular subspaces of $U$,
  since $k-1$ is even. That is, there are $6$ totally singular ($k-1$)-dimensional subspaces $U_1,\ldots,U_6$ of $U$, 
  which form a base for the action of $G_U$ on $U$. 
  By construction, 
  \[
  U_1=\langle x_1,\ldots,x_{k-1}\rangle,\quad
  U_2=\langle y_1,\ldots,y_{k-1}\rangle
  \]
  with $[x_i,y_j]=\delta_{ij}$ for $1\leq i,j \leq k-1$. 
  Define the subspaces $V_s=\langle x\rangle\oplus U_s$ for 
  $1\leq s\leq 6$.
  Furthermore, let
  \[
  W_1=\langle y,y_1,x_2,\ldots,x_{k-1}\rangle,\quad
  W_2=\langle y,x_1,y_2,\ldots,y_{k-1}\rangle.
  \]
  Then all of $V_1,\ldots,V_6,W_1,W_2$ are totally singular
  $k$-dimensional subspaces with pairwise odd-dimensional
  intersections, so they are in the same $G$-orbit $X$.  Adding
  $V_1,\ldots, V_6$, $W_1$, $W_2$ to $\mB$ we see that any $g\in G_{(\mB)}$
  fixes the subspaces $\langle x\rangle=V_1\cap V_2,\, \langle
  y\rangle=W_1\cap W_2$ and $U=(V_1+V_2)\cap (W_1+W_2)$.  Furthermore,
  $g\in G_{(\mB)}$ also fixes $U_s=V_s\cap U$ for each $s$, so $g_U$ is a
  scalar transformation by the definition of the $U_s$.  Adding also
  the subspace
  \[
  W_3=\langle x+x_1,y-y_1,x_2,\ldots,x_{k-1}\rangle \in X
  \]
  to $\mB$, we get that any $g\in G_{(\mB)}$ is a scalar transformation on
  the whole of $V$. Hence $b_X(G)\leq 9$ in this case. 
\end{proof}

\begin{rem}
  By a more detailed argument, Burness, Guralnick and Saxl were able to
  calculate the exact base size for a classical group over an 
  algebraically closed field acting on an orbit of subspaces of its
  natural module \cite[Section 4]{BGS2}.  While part of their
  constructions could be translated to the finite case, we had to give
  new constructions for other cases. (This is especially true for the
  orthogonal case, since, in contrast to the finite case discussed
  above, in any dimension there is just one type of non-degenerate
  orthogonal space over an algebraically closed field.)
\end{rem}

\subsection{Action on pairs of subspaces}
\label{transpose}
In this subsection we handle the subspace actions arising from case (2) in the list after the statement of Theorem \ref{generalclassical}.

\begin{prop}\label{pairs}
Let $G = PSL(V) = PSL_d(q)$, and let $M$ be the stabilizer in $G$ of a pair $U,W$ of nonzero subspaces, where
$\dim U = k < d/2$, $\dim W = d-k$, and  either $U\subseteq W$ or $V = U\oplus W$. Let $X$ be the coset space $G/M$. 
Then 
\[
b_{X}(G) \leq \frac{d}{k} + 11 \leq 2\frac{\log |G|}{\log |X|} + 12.
\]
\end{prop}

\begin{proof}
Let $X_k$ be the set of all $k$-dimensional subspaces of $V$. A straightforward computation shows that $|X| < |X_k|^2$. Clearly $b_X(G) \le b_{X_k}(G)$. Now the result follows from  Theorem \ref{classicalmain} and Proposition \ref{prop:ClassicalBound_nperk}.  
\end{proof}

\subsection{Proof of Theorem \ref{generalclassical}}
\label{Sec3.3}

Let $G$ be an almost simple group with socle $G_0$, a classical group on $V$, a vector space of dimension $d$ over $\F_q$. Suppose $G$ acts faithfully and primitively on a set $\O$. 

If the action of $G$ on $\O$ is not a subspace action, then $b(G) \le 5$ by \cite{Burness}. Hence we may assume that the action is a subspace action, so that one of the cases (1), (2), (3) listed after the statement of Theorem \ref{generalclassical} holds. 

In case (1), $\O = U^G$ is an orbit of $G$ on $k$-dimensional subspaces, for some $k$, and we can assume that $k\le d/2$ (by replacing $U$ with $U^\perp$ if necessary, in the case where $G_0 \ne PSL(V)$, and by considering the equivalent action of $G$ on $(d-k)$-spaces, when $G_0 = PSL(V)$). Now Theorem \ref{classicalmain} and Proposition \ref{prop:ClassicalBound_nperk} give
\[
b(G_0) \le \frac{d}{k}+11 \le 2 \frac{\log |G|}{\log |\O|} + 13.
\]
Hence we can choose a set $\mB$ of at most $\frac{d}{k}+11$ points of $\O$ such that $G_{(\mB)} \cap G_0 = 1$, so that $G_{(\mB)}$ is isomorphic to a subgroup of $G/G_0$. This is a soluble group possessing a normal series of length at most 3 with cyclic factor groups. Since the base size of a cyclic linear group is 1, by \cite{SZ}, it follows that 
$b(G) \le 2 \frac{\log |G|}{\log |\O|} + 16$, as required.

Now consider case (2): here $G_0 = PSL(V)$ and $\O = \{U,W\}^G$ where $U,W$ are subspaces of dimensions $k,\,d-k$ and either $U\subseteq W$ or $V = U\oplus W$. In the latter case, if $k = d/2$ then $G_0$ has an element interchanging $U$ and $W$, and $(G,\O)$ is not a subspace action (it is a ${\mathcal C}_2$-action in the terminology of \cite{Burness}). Hence we may assume that $k<d/2$. Now Proposition \ref{pairs} implies that $b(G_0) \le 2 \frac{\log |G|}{\log |\O|} + 12$, and this yields the result as above.

Finally, consider case (3): here $G_0 = Sp_{2m}(q)$, $p=2$ and $M\cap G_0 = O_{2m}^\pm(q)$, where $M$ is a point-stabilizer in $G$. Regarding $G_0$ as the isomorphic orthogonal group $O_{2m+1}(q)$, the set $\O$ is an orbit of $G_0$ on hyperplanes of the natural module $V_{2m+1}(q)$. Hence $b_{\O}(G_0) \le 2m+1$, which is less than 
$2 \frac{\log |G|}{\log |\O|} + 3$, and the conclusion follows again. 

This completes the proof of Theorem \ref{generalclassical}.

\section{Non-affine primitive permutation groups}

 In this section we prove the main Theorem \ref{mainresult} for primitive groups which are not of affine type.

\begin{thm}
\label{nonaffine}
Let $G$ be a primitive permutation group of degree $n$. Assume that $G$ is not of affine type. Then $b(G) \leq 2 (\log |G| / \log n) + 24$. 
\end{thm}

According to the O'Nan-Scott theorem (see for example \cite{LPS}), non-affine primitive groups are of the following types: 
almost simple, diagonal type, product type, and twisted wreath type. We shall deal with these types separately in the following subsections.

\subsection{Almost simple groups}

For this case we prove

\begin{thm}
\label{almostsimple}
Let $G$ be an almost simple primitive permutation group of degree $n$. Then $b(G) \leq 2 (\log |G| / \log n) + 16$. 
\end{thm}

\begin{proof}
Theorems \ref{generalalternating} and \ref{generalclassical} give the result when the socle of $G$ is an alternating or classical group. For the remaining cases, the socle of $G$ is a group of exceptional Lie type or a sporadic group. In these cases we have $b(G) \leq 7$ by \cite{BurnessLiebeckShalev} and \cite{BurnessObrienWilson}. 
\end{proof}

\subsection{Diagonal type groups}
\label{Sec4.2}

Work of Fawcett \cite{Fawcett} (and also Gluck, Seress, Shalev \cite[Remark 4.3]{GSS}) implies that, in the diagonal type case, we have 
\begin{equation}
\label{equat}
b(G) \leq (\log |G| / \log n) + 3.
\end{equation}

\subsection{Product type groups}

Bases for primitive groups of product type were studied by Burness and Seress in \cite{BS}. 
We will use their notation. Let $\Omega
= \Gamma^k$ for some set $\Gamma$ and integer $k \geq 2$. There exists
a primitive group $H \leq \mathrm{Sym}(\Gamma)$ of almost simple type or of
diagonal type such that the following holds. Let the socle of $H$ be
$T$. Let $P$ be the (transitive) action of $G$ on the set of the $k$
direct factors of $\mathrm{Soc}(G) = T^k$. We have $T^{k} \leq G \leq H \wr
P$.  

We recall two definitions. A {\it distinguishing partition} for a finite group $X$ acting on a finite set $\Sigma$ is a coloring of the
points of $\Sigma$ in such a way that every element of $X$ fixing this
coloring is contained in the kernel of the action of $X$ on
$\Sigma$. The minimal number of parts (or colors) of a distinguishing
partition is called the {\it distinguishing number} of $X$ and is denoted by
$d(X)$. 

Let $d(P)$ be the distinguishing number of the transitive permutation group $P$. By \cite[Theorem 1.2]{DHM}, we have $d(P) \leq 48 \sqrt[k]{|P|}$.


\subsubsection{The case when $H$ is almost simple}
\label{Sec4.3.1}

Assume that $H \leq \mathrm{Sym}(\Gamma)$ is an almost
simple group with socle $T$. We follow not only \cite{BS} here but \cite[\S4]{DHM}. 
However we avoid the use of the bound $|\mathrm{Out}(T)| \leq {|T|}^{\alpha}$, since this is expensive. Instead we use the estimate $|\mathrm{Out}(T)| \leq |\Gamma|$ found in \cite[Lemma 2.7]{AG}. Thus $|G| \geq
{|T|}^{k}|P| \geq ({|H|}^{k}|P|)/{|\Gamma|}^{k}$. This gives 
$\log({|H|}^{k}|P|)/\log |\Omega| \leq (\log |G| / \log |\Omega|) + 1$.

By using the idea of \cite[Lemma 3.8]{BS} combined with Lemma 2.1 of \cite{DHM}, we see that 
\begin{equation}
\label{e11}
b(G) < \frac{\log d(P)}{\log |\Gamma|} + 1 + b(H) < \frac{\log |P|}{\log |\Omega|} + b(H) + 4,
\end{equation}
 since $|\Gamma| \geq 5$. By Theorem \ref{almostsimple}, this gives $$b(G) < \frac{\log |P|}{\log |\Omega|} + 2 \frac{\log |H|}{\log |\Gamma|} + 20 <
2  \frac{\log ({|H|}^{k}|P|)}{\log |\Omega|} + 20 \leq 2 \frac{\log |G|}{\log |\Omega|} + 22.$$

\subsubsection{The case when $H$ is of diagonal type} 
\label{Sec4.3.2}

Now assume that $H$ is of diagonal type. Here $\mathrm{Soc} (H) = T =
S^{\ell}$, where $S$ is a non-abelian simple group and $\ell \geq
2$. We have $S^{\ell} \leq H \leq S^{\ell}.(\mathrm{Out}(S) \times Q)$
where $Q \leq \mathrm{Sym} (\ell)$ is the permutation group induced by the
conjugation action of $H$ on the $\ell$ factors of $S^{\ell}$. 

The set $\Gamma$ can be thought of as the set of right cosets in $H$ of the subgroup $H_{0} = (D \times Q) \cap H$ where $D$ denotes the diagonal subgroup of ${\mathrm{Aut}(S)}^{\ell}$. In particular, $|\Gamma| = {|S|}^{\ell-1}$. By (\ref{equat}), $b(H) \leq (\log|H|/\log|\Gamma|) + 3 \leq 8$, provided that $\ell \leq |S|$. Thus, in view of (\ref{e11}), we may assume that $\ell \geq 3$.  

Let $\mathcal{C}$ be the set of complete representatives of the right cosets in $H$ of the subgroup $H_{0}$ consisting of the elements of $S^{\ell}$ where the first coordinate is $1$. Let $s_1$ and $s_2$ be elements of $S$ such that they together generate $S$. Let $\gamma_{0}$, $\gamma_{1}$, $\gamma_{2} \in \mathcal{C}$ be those elements for which every coordinate of $\gamma_{0}$ is $1$, all but the first coordinate of $\gamma_{1}$ is $s_{1}$, and all but the first coordinate of $\gamma_{2}$ is $s_{2}$. Consider the pointwise stabilizer $Q_0$ of $\{ H_{0}\gamma_{0}, H_{0}\gamma_{1}, H_{0}\gamma_{2} \}$. This group $Q_0$ is contained in the stabilizer $H_0$ of $H_{0}\gamma_{0}$. For any element $h_{0} \in H_{0}$ and any index $i$ in $\{1, 2\}$, we have $H_{0}\gamma_{i}h_{0} = H_{0}\gamma$ for some $\gamma \in \mathcal{C}$ with $1$ or $\ell -1 > 1$ entries equal to $1$. Moreover if $h_{0}$ is in $Q_{0}$, then the first case must hold. Since the only automorphism of $S$ fixing both $s_{1}$ and $s_{2}$ is the identity, we see that $Q_{0}$ is a subgroup of $Q$ leaving $\mathcal C$ invariant. Therefore, whenever $q_0\in Q_0$ and $\gamma\in \mathcal C$, we have $(H_0\gamma)q_0=H_0\gamma^{q_0}$ for $\gamma^{q_0}\in \mathcal C$.

Let $\omega_{0}$, $\omega_{1}$, $\omega_{2}$ be those elements of $\Omega$ for which all $k$ coordinates of $\Omega$ are $H_{0} \gamma_{0}$, $H_{0} \gamma_{1}$, $H_{0} \gamma_{2}$, respectively. By the previous paragraph and the fact that $G \leq H \wr P$, the pointwise stabilizer in $G$ of $\{ \omega_{0}, \omega_{1}, \omega_{2} \}$ is a permutation group $R$ permuting the $k \ell$ coordinates of the vectors in $S^{k \ell}$. More precisely, if the coordinates are labelled by the integers $1, \ldots , k \ell$, then $R$ is a permutation group on $\{ 1, \ldots, k \ell \}$ such that $\{ j \ell + 1 \mid j \in \{0, \ldots , k-1 \} \}$ is $R$-invariant. Since $R$ is a subgroup of a transitive group on $k \ell$ points which has order at most $|G|$, we see, by \cite[Theorem 1.2]{DHM}, that $d(R) \leq 48 \sqrt[k \ell]{|G|}$. 

Consider a distinguishing partition $\mathcal{P}$ with $d(R)$ colors for the action of $R$ on $\{ 1, \ldots, k \ell \}$. Define a new coloring of the $R$-invariant subset $\{ 1, \ldots, k \ell \} \setminus \{ j \ell + 1 \mid j \in \{0, \ldots , k-1 \} \}$ using no more than $d(R)^{2}$ colors in the following way. For any integers $j$ and $u$ with $0 \leq j \leq k-1$ and $1 < u \leq \ell$ color $j \ell + u$ with the color $(\alpha, \beta)$ where $\alpha$ is the color of $j \ell + 1$ in $\mathcal{P}$ and $\beta$ is the color of $j \ell + u$ in $\mathcal{P}$. Clearly, no non-identity element of $R$ preserves this new coloring. For $\ell \geq 3$, we see, by Lemma 2.1 of \cite{DHM}, that $G$ has a base $B$ containing $\{ \omega_{0}, \omega_{1}, \omega_{2} \}$ such that
$$b(G) \leq |B| = 2 \frac{\log d(R)}{\log |S|} + 4 < 2 \frac{\log |G|}{k \ell \log |S|} + 6 = 2 \frac{\log |G|}{\log n}  + 6.$$

\subsection{Twisted wreath product type groups} 
\label{Sec4.4}

This type was treated in Burness and Seress \cite[Section 4]{BS}. We follow their discussion. By the previous section we know that if $L$ is a primitive permutation group of product type acting on a set $\Omega$, then we have $b(L) \leq 2 (\log |L| / \log |\Omega|) + 22$. Let $G$ be a primitive permutation group of twisted wreath product type acting on the set $\Omega$. Then $G$ contains a regular normal subgroup $T^k$ isomorphic to the direct product of $k$ copies of a non-abelian simple group $T$. We may write $G = T^{k} P$ where $P$ is a transitive permutation group acting on $k$ points. As explained in \cite[Section 3.6]{P}, we may embed $G$ in a group of product type $L$ which is of the form ${(T^{2})}^{k}.P$. Thus $b(G) \leq b(L) \leq 2 (\log |L| / \log |\Omega|) + 22 = 2 (\log |G| / \log |\Omega|) + 24$. 
 
\vspace{4mm}
This completes the proof of Theorem \ref{nonaffine}.

\section{Affine primitive permutation groups}

The main result of this section is

\begin{thm} 
\label{generalaffine}
Let $G$ be an affine primitive permutation group of degree $n$. Then $b(G) \leq 2 (\log |G| / \log n) + 16$.
\end{thm}

Let $G$ be an affine primitive permutation group of degree $n$ with a point-stabilizer $H$. Then $G = VH \le AGL(V)$, where $V$ is a finite vector space of order $n = p^k$ ($p$ prime), and the stabilizer $H$ of the zero vector is an irreducible subgroup of $GL(V)$. Since $b(G) = b(H)+1$, Theorem \ref{generalaffine} follows immediately from 

\begin{thm}
\label{ez}
Let $H$ be a subgroup of $GL(V)$ acting irreducibly on the finite vector space $V$. Then $b_V(H) \leq 2 (\log |H| / \log |V|) + 17$.
\end{thm}

In the above theorems, the multiplicative constant $2$ is best possible, as shown by the following example (which completes the proof of Proposition \ref{acc}). 

\begin{prop}
\label{sp}
Let $V$ be a $d$-dimensional ($d$ even) non-degenerate symplectic space over 
the finite field $\FF q$ and let $H= Sp(V)$ with its natural 
action on $V$. 
\begin{itemize}
\item[{\rm (i)}] Then $b_V(H)=d$. 
\item[{\rm (ii)}] If $G = VH \le AGL(V)$ is the corresponding affine primitive permutation group, then for sufficiently large values of $q$, we have 
\[
b(G) = \lfloor 2 (\log |G|/\log n) \rceil - 2.
\]
\end{itemize}
\end{prop}
\begin{proof}
(i) Clearly, any basis of $V$ (as a vector space) is also a
base for $H$, so $b_V(H)\leq d$. For the equality, let 
$\{b_1,\ldots,b_l\}\in V$ be any set of vectors with $l\leq d-1$. 
Then there is a subspace $U\leq V$ containing $\{b_1,\ldots,b_l\}$ with 
$\dim U=d-1$. Hence it is enough to show that for every such subspace $U$, 
there exists  a non-identity $g\in H$ that acts trivially on $U$. 

Let $U\leq V$ be a subspace of dimension $d-1$ and let $[\ ,\ ]$
denote the non-degenerate symplectic bilinear form on $V$ preserved by $H$. Then
the restriction of $[\ ,\ ]$ to $U$ is degenerate: there exists 
$0\neq x\in U$ such that $\langle x\rangle=U^\perp$. Let $y\in
V\setminus U$ be arbitrary. We claim that the map
\[
A:cy+u\mapsto c(y+x)+u\ (c\in\FF q,\,u\in U)
\]
 is an element of $H$, which acts 
trivially on $U$. To see this, let $c,d\in \FF q$ and $u,v\in U$. Then we have
$[A(cy+u),A(dy+v)]=[cy+u+cx,dy+v+dx]=[cy+u,dy+v]$, proving the claim.

(ii) This follows from a simple computation using the order formula for $|Sp(V)|$.
\end{proof}

It remains to prove Theorem \ref{ez}. We do this in the following two subsections.

\subsection{Primitive linear groups}

In this subsection we prove Theorem \ref{ez} in the case where $H\le GL(V)$ acts primitively on $V$ as a linear group. In fact we prove the following stronger bound for this case. 

\begin{thm}\label{main}
Let $V$ be a finite vector space, and let $H \le GL(V)$ be an irreducible, primitive linear group on $V$. Then one of the following holds:
\begin{itemize}
\item[{\rm (i)}] $b(H) \leq 15$;
\item[{\rm (ii)}] $b(H) \le 2\,\frac{\log |H|}{\log |V|} + 9$.
\end{itemize}
\end{thm}

A version of Theorem \ref{main} was proved in \cite{LSbase, LSbase2} with a much worse multiplicative constant, and  unspecified constants in place of the constants 15 and 9.  
The proof of Theorem \ref{main} will be along the lines of that proof. However, in order to make our constants explicit 
(and small), we need to improve several of the results in \cite{LSbase, LSbase2}. 

For a field $\F_q$ and a positive integer $m$, by the {\it natural} module for the symmetric or alternating group $\mathrm{Sym}(m)$ or $\mathrm{Alt}(m)$ over $\F_q$, we mean the fully deleted permutation module of dimension $m' = m-\d$, where $\d \in \{1,2\}$.

The first result is a version of Proposition 2.2 of \cite{LSbase} with an explicit constant:

\begin{prop} \label{quasi}
     Let $V = V_d(q)$ $(q=p^e)$ and $G \leq GL(V)$, and suppose that $E(G)$ is 
     quasisimple and absolutely irreducible on $V$. Then one of the following 
     holds: 
\begin{itemize}
\item[{\rm (i)}] $E(G) = \mathrm{Alt}(m)$ and $V$ is the natural $\mathrm{Alt}(m)$-module 
     over $\F_q$, of dimension $d = m-\d$ ($\d \in \{1,2\}$); 
\item[{\rm (ii)}] $E(G) = Cl_d(q_0)$, a classical group  with natural module 
of dimension $d$ over a subfield $\F_{q_0}$ of $\F_q$;
\item[{\rm (iii)}] $b(G)\le 6$. 
\end{itemize}
\end{prop}

\begin{proof}
This is proved in \cite{LL}. 
\end{proof}

The next result is an explicit version of \cite[Proposition 3.6]{LSbase}.

\begin{prop}\label{fitt}
 Let $V = V_d(q)$ $(q=p^e)$ and $G \leq GL(V)$, and suppose that the Fitting subgroup $F(G)$ is absolutely irreducible on $V$. Then $b(G) \le 13$.
\end{prop}

\begin{proof}
We begin by arguing exactly as in the proof of \cite[3.6]{LSbase} that $F=F(G)$ can be taken to be the central product of an extraspecial group $s^{1+2m}$ and the group $Z = \F_q^*$ of scalars, where $s$ is a prime and $d = s^m$. We can also assume that $G = N_{GL(V)}(F)$, so that $G/F$ is isomorphic to either $Sp_{2m}(s)$ or $O_{2m}^\pm (2)$, with $s=2$ in the latter case. Moreover, $q\equiv 1\hbox{ mod }s$, and also $q\equiv 1\hbox{ mod }4$ if $s=2$ and $G/F \cong Sp_{2m}(2)$.

Define $F^0 = F.Z(G/F)$, an extension of $F$ by a group of order at most 2.
By \cite{Se}, there are three vectors $v_1,v_2,v_3 \in V$ such that $F^0_{v_1v_2v_3}=1$. So if we let $J = G_{v_1v_2v_3}$, then $J\cap F^0=1$ and $J \cong JF^0/F^0$ is isomorphic to a subgroup of $PSp_{2m}(s)$ or $O_{2m}^\pm(2)$.

Obviously $b(G) \le 3+b(J)$. Assume for a contradiction that 
\[
b(J) > 10.
\]
Then clearly $d = \dim V > 10$, and $V^{10} = \bigcup_{h\in J\setminus 1}C_{V^{10}}(h)$. Hence
\begin{equation}\label{triv}
|V|^{10} \le \sum_{h\in J\setminus 1} |C_V(h)|^{10}.
\end{equation}
For $h\in J\setminus 1$, Theorem 4.1 of \cite{GS} shows that there are $2m+1$ conjugates of $h$ that generate $G$ modulo $F$, and hence there are $2m+2$ conjugates of $h$ generating $G$. It follows that 
\[
\dim C_V(h) \le \left(1-\frac{1}{2m+2}\right)\dim V.
\]
Hence (\ref{triv}) gives $|V|^{10/(2m+2)} \le |J|$, and so as $|V| = q^d = q^{s^m}$, we have 
\begin{equation}\label{bd}
q^{10s^m/(2m+2)} \le |J| \le \left\{\begin{array}{l}
|PSp_{2m}(s)|,\,s \hbox{ odd}, q\equiv 1\hbox{ mod }s \\
|Sp_{2m}(2)|,\,s=2,\, q\equiv 1\hbox{ mod }4 \\
|O^\pm_{2m}(2)|,\,s=2,\, q \hbox{ odd}.
\end{array}
\right.
\end{equation}
Straightforward computation shows that the only possible values satisfying (\ref{bd}) are $s=2$, $q=3$ and $m=4$ or 5. For $m=5$ we have $G/F \cong O_{10}^\pm (2)$, and in the above argument, \cite[4.1]{GS} shows that we may replace $2m+2$ by $2m+1$ in (\ref{bd}), yielding a contradiction. And if $m=4$ then $d = 2^4=16$ and it is very easy to argue directly that $b(G) \le 13$, as required. 
\end{proof}

The next result is an improvement of Lemma 3.7 of \cite{LSbase}. 

\begin{prop}\label{sub} {\rm (i)} Let $\F_{q_0}$ be a subfield of $\F_q$, let $q = q_0^r$, and let $M = \F_q^*\,GL_d(q_0) \le GL_d(q) = GL(V)$, where $\F_q^*$ denotes the group of scalars. Then for the action of $M$ on $V$ we have
\[
b(M) \le \frac{d}{r}+2.
\]
{\rm (ii)} Let $q=p^r$ with $p$ prime, and let $M = \F_q^*\, \mathrm{Sym}(m) \le GL_{m'}(q) = GL(V)$, where $V$ is the natural module for $\mathrm{Sym}(m)$ over $\F_q$, of dimension $m' = m-\d$ ($\d\in\{1,2\}$). Then 
\[
b(M) \le \frac{\log_pm}{r}+4.
\]
\end{prop}

\begin{proof} (i)  Let $\l_1,\ldots ,\l_r$ be an $\F_{q_0}$-basis for $\F_q$, and let $e_1.\ldots,e_d$ be the standard basis for $V = \F_q^d$ (that is, $e_i = (0,\ldots ,1,\ldots 0)$ where the 1 is in the $i^{th}$ coordinate). Write $d=kr+l$ with $k,l \in \Z$ and $0\le l<r$, and define
\[
\begin{array}{l}
v_i = \sum_{j=(i-1)r+1}^{ir} \l_je_j\;\;(1\le i\le k), \\
v_{k+1} = \sum_{j=kr+1}^d \l_je_j.
\end{array}
\]
If we write $M_0 = GL_d(q_0)$, it is easy to see that $(M_0)_{v_1\ldots v_{k+1}} = 1$. Hence if 
$J = M_{v_1\ldots v_{k+1}}$ then $J \cong JM_0/M_0$ is cyclic, and so by \cite[3.1]{SZ}, $J$ has a base of size 1. Thus $b(M) \le k+2 \le \frac{d}{r}+2$.

(ii) This follows directly from the proof of \cite[3.7(ii)]{LSbase}.
\end{proof}

As in \cite{LSbase}, for $H\le GL(V)$ define $b^*(H)$ to be the minimal size of a set $B$ of vectors such that any element of $H$ that fixes every 1-space $\la v\ra$ with $v\in B$ is necessarily a scalar multiple of the identity. We call such a set $B$ a {\it strong} base for $H$. By \cite[3.1]{LSbase}, 
\[
b(H) \le b^*(H) \le b(H)+1.
\]
Next we give an improvement of Lemma 3.3(iii) of \cite{LSbase}. 

\begin{lem}\label{33}
Let $V_1,V_2$ be vector spaces over $\F_q$ with $\dim V_i = n_i$ 
     and $n_1 \leq n_2$, and let $H_i \leq GL(V_i)$ for $i=1,2$. Denote 
     by $H_1 \otimes H_2$ the image of  $H_1 \times H_2$ acting in the 
     natural way on the tensor product $V_1 \otimes V_2$. 

If $n_1 \le b^*(H_2)$, then 
     \[ 
     b(H_1 \otimes H_2) \le \frac{b^*(H_2)}{n_1}+3. 
     \] 
\end{lem}

\begin{proof}
We follows the proof of \cite[Lemma 3.3(iii)]{LSbase}. Let $b = b^*(H_2)$. Assume $n_1\le b$, and let 
$y_1, \ldots ,y_b$ be a  linearly independent strong base for $H_2$ in $V_2$. 
     Let $x_1,\ldots ,x_{n_1}$ be a basis of $V_1$. 

     Write $b = rn_1+s$ with $r,s$ integers and $0 \leq s < n_1$ 
     For $1 \leq i \leq r$ define 
     \[ 
     v_i = \sum_{k=1}^{n_1} x_k \otimes y_{(i-1)n_1+k}, \;\;W_i = \la x_k \otimes y_{(i-1)n_1+k} : 1\le k\le n_1\ra,
     \] 
     and set $v_{r+1} = \sum_{k=1}^s x_k \otimes y_{rn_1+k}$, $W_{r+1} =  \la x_k \otimes y_{rn_1+k} : 1\le k\le s\ra$. 

     Consider the stabilizer $L=(H_1 \otimes H_2)_{v_1\dots v_{r+1}}$. By Lemma 3.3(i) of \cite{LSbase}, $L$ stabilizes $V_1\otimes W_i$ for all $1\le i\le r+1$. 

Next choose $C,D \in SL_{n_1}(q)$ generating $SL_{n_1}(q)$, and for each $i$, define $\g_i = C\otimes 1, \d_i=D\otimes 1 \in GL(V_1 \otimes W_i)$. Let
\[
v = \sum_{i=1}^{r+1}v_i\g_i,\; w = \sum_{i=1}^{r+1}v_i\d_i.
\]
At this point the argument at the end of the proof of \cite[3.3(iii)]{LSbase} shows that $L_{vw}=1$. Hence 
$ b(H_1 \otimes H_2)  \le r+3 \le \frac{n_1}{b}+3$, as required. 
\end{proof}

The next result is Theorem 1 of \cite{LSbase2}, with an explicit constant $C = 14$. The proof is identical to that in \cite{LSbase2}, but using Propositions \ref{quasi} and \ref{fitt} at the end to justify that $C=14$ works.

 \begin{prop}\label{corr}
Let $V = V_d(q)$, and let $H$ be a subgroup of $\G L(V)$ such that $H$ acts primitively on $V$ and $H^0:=H\cap GL(V)$ is absolutely irreducible on $V$.  Suppose that $b^*(H^0) > 14$. Then 
\[
H^0 \le H_0 \otimes \bigotimes_{i=1}^s \mathrm{Sym}(m_i) \otimes \bigotimes_{i=1}^t {\rm Cl}_{d_i}(q_i),
\]
where $s+t\ge 1$ and the following hold:
\begin{itemize}
\item[(i)] $H_0\le GL_{d_0}(q)$ with $b^*(H_0)\le 14$
\item[(ii)] each factor $\mathrm{Sym}(m_i) < GL_{m_i'}(q)$ and ${\rm Cl}_{d_i}(q_i) \le GL_{d_i}(q)$ is acting on the natural module over $\F_q$, where $m_i' = m_i-\d_i$, $\d_i\in \{1,2\}$
\item[(iii)] $d = d_0\cdot \prod_1^s m_i' \cdot \prod_1^t d_i$
\item[(iv)] $F^*(H^0)$ contains $\prod_1^s \mathrm{Alt}(m_i) \cdot \prod_1^t {\rm Cl}_{d_i}(q_i)^{(\infty)}$.
\end{itemize}
\end{prop}

The next result is an improvement of \cite[Proposition 2]{LSbase2}.

\begin{prop} \label{basest}
Let $H,H^0$ be as in Proposition $\ref{corr}$, with $b^*(H^0)>14$. Take $m_s' = {\rm max}(m_i':1\le i\le s)$ and $d_t = {\rm max}(d_i:1\le i\le t)$ (define these to be $0$ if $s=0$ or $t=0$, respectively).
\begin{itemize}
\item[(i)] Suppose $t \ge 1$ and $m_s' \le d_t$, and let $q = q_t^r$. Then $d<d_t^2$, $d_t \ge 14$, and
\[
b(H^0) \le b(GL_{d/d_t}(q) \otimes GL_{d_t}(q_t)) \le \frac{d_t^2}{dr}+5.
\]
\item[(ii)] Suppose $s \ge 1$ and $m_s' > d_t$, and let $q = p^r$. Then $d<(m_s')^2$, $m_s' \ge 14$, and
\[
b(H^0) \le b(GL_{d/m_s'}(q) \otimes \mathrm{Sym}(m_s)) \le \frac{m_s\log_p m_s}{dr}+6.
\]
\end{itemize}
\end{prop}

\begin{proof}
We follow the proof of \cite[Proposition 2]{LSbase2}, but as the constants are different we give a few details. 

We proceed by induction on $s+t$. For the base case $s+t=1$, we have $H^0 \le H_0\otimes M$, where 
$M = {\rm Cl}_{d_1}(q_1)$ or $\mathrm{Sym}(m_1)$. Consider the first case, and write $q=q_1^r$. Proposition \ref{sub} gives 
$b(M) \le \frac{d_1}{r}+2$, hence also $b^*(M) \le \frac{d_1}{r}+3$. If $d_0=1$ the conclusion in (i) is immediate, so assume $d_0\ge 2$. As in the proof of \cite[Proposition 2]{LSbase2}, we see that $d_0\le d_1$.  Then Lemma \ref{33} gives
\[
\begin{array}{ll}
b(H^0) & \le \frac{b^*(M)}{d_0}+3 \\
             & \le \frac{d_1}{rd_0}+\frac{3}{d_0}+3 \\
             & < \frac{d_1}{rd_0}+5,
\end{array}
\]
so that (i) holds (note that $d_1\ge 14$ by
\cite[3.3(ii)]{LSbase}). Similarly (ii) holds when $M = \mathrm{Sym}(m_1)$. 

Now assume $s+t\ge 2$. Let $m$ be the maximum of $d_t$ and $m_s'$, and write $M$ for the corresponding group ${\rm Cl}_{d_t}(q_t)$ or $\mathrm{Sym}(m_s)$. Note that $m\ge 14$, since 
otherwise \cite[3.3(ii)]{LSbase}  implies that $b^*(H^0) \le 14$. 
Let $N$ be the tensor product of $H_0$ and the other factors 
 ${\rm Cl}_{d_i}(q_i)$, $\mathrm{Sym}(m_i)$, so that $H^0 \le N\otimes M$. If $b^*(N)\le 14$ the conclusion follows as in the $s+t=1$ case, so assume $b^*(N)>14$.

Let $m'$ be the largest among the dimensions $d_i,m_i'$ omitting $m$, and write $N_1$ for the corresponding group 
 ${\rm Cl}_{d_i}(q_i)$ or $\mathrm{Sym}(m_i)$. 

Consider the case where  $N_1 = {\rm Cl}_{d_i}(q_i)$. Let $q = q_i^u$. 
By induction we have 
\[
b^*(N) \le b(N)+1 \le \frac{d_i^2m}{du}+6 \le \frac{d_i}{u}+6.
\]
Suppose $d\ge m^2$. Then $b^*(N)\ge m$ by \cite[3.3(iv)]{LSbase}, so Lemma \ref{33} implies that 
\[
b(H^0) \le  \frac{b^*(N)}{m}+3 \le \frac{d_i}{um} + \frac{6}{m}+3.
\]
Since $m\ge d_i$ and $m>14$, this yields 
$b(H^0) < 5$, a contradiction. Hence $d<m^2$ in this case. Now the conclusion of the proposition follows by the argument given for the $s+t=1$ case.

Finally, consider the case where $N_1 = \mathrm{Sym}(m_i)$. Let $q=p^r$. By induction, 
\[
b^*(N) \le b(N)+1 \le  \frac{(m_i\log_pm_i)\cdot m}{dr}+7 \le \frac{\log_pm_i}{r}+8.
\]
Now the argument of the previous paragraph gives the conclusion.
\end{proof}

\vspace{4mm}
\noindent {\bf Proof of Theorem \ref{main}}

We are now in a position to prove Theorem \ref{main}.

We begin just as in the proof of \cite[Corollary 3]{LSbase2}. Suppose $H\le GL(V)$ acts primitively and irreducibly on a finite vector space $V$ defined over a field of size $q_{0}$. Choose $q=q_0^r$ maximal such that $H \le \G L_d(q) \le GL_{dr}(q_0)$. Write $H^0 = H\cap GL_d(q)$ and $V = V_d(q)$. By \cite[12.1]{TAMS}, $H^0$ is absolutely irreducible on $V$.  

If $b^*(H^0) \le 14$ then $b(H) \le 15$, and the conclusion of Theorem \ref{main} holds. So assume now that 
$b^*(H^0) > 14$. Then $H^0$ is given by Proposition \ref{corr}, and (i) or (ii) of Proposition \ref{basest} holds. 

Consider case (i) of Proposition \ref{basest}. Write $m=d_t$ and $q = q_t^r$. Then $d<m^2$, $H^0 \triangleleft Cl_m(q_t)$, $m\ge 14$,  and 
\begin{equation}\label{last}
b(H^0) \le \frac{m^2}{dr}+5.
\end{equation}
From the order formulae for classical groups, we see that 
\[
\log_{q_t}(|H^0|) \ge \log_{q_t} |\O_m^\pm (q_t)| \ge \frac{1}{2}m(m-1)-1.
\]
Hence 
\[
\frac{\log |H|}{\log |V|} \ge \frac{\frac{1}{2}m(m-1)-1}{rd}, 
\]
and so (\ref{last}) gives
\[
b(H^0) \le 2\,\frac{\log |H|}{\log |V|} + 5 + \frac{m+2}{rd} < 2\,\frac{\log |H|}{\log |V|} + 7.
\]
This completes the proof  in case (i) of Proposition \ref{basest}.

To conclude, consider case (ii) of Proposition \ref{basest}. Write $m=m_s$, $m'=m_s'$ and $q = p^r$. Then $d<m^2$, $H^0 \triangleleft \mathrm{Alt}(m)$, $m'  \ge 14$,  and 
\begin{equation}\label{symeq}
b(H^0) \le \frac{m\log_p m}{dr}+6.
\end{equation}
Now $|H^0| \ge \frac{1}{2}m! > \frac{1}{2}\left(\frac{m}{e}\right)^m$, and so 
\[
\frac{\log |H|}{\log |V|} \ge \frac{m(\log_pm-\log_pe)-1}{rd}. 
\]
Hence (\ref{symeq}) gives
\[
b(H^0) \le \frac{\log |H|}{\log |V|} + 6 + \frac{m\log_2e+3}{rd} < \frac{\log |H|}{\log |V|} + 8,
\]
giving the conclusion of Theorem \ref{main} (actually, a stronger version, with the constant 2 replaced by 1). 

This completes the proof of Theorem \ref{main}. 

\subsection{Imprimitive linear groups}

It remains to prove Theorem \ref{ez} in the case where the irreducible linear group $H \le GL(V)$ acts imprimitively on $V$.
Assume that $H$ preserves the direct sum decomposition $V = V_{1} \oplus \cdots \oplus V_{t}$ where $V_{1}, \ldots , V_{t}$ are subspaces of $V$, and $t>1$ is chosen maximally. Let $H_{1}$ be the stabilizer of $V_{1}$ in $H$. Let us denote the minimal base size of the action $K_{1}$ of $H_{1}$ on $V_{1}$ by $b(K_{1})$. By the choice of $t$, the action of $K_1$ on $V_1$ is primitive, so by Theorem \ref{main}, we have $b(K_{1}) \leq 15$ or 
$$b(K_{1}) \leq 2 (\log |K_{1}| / \log |V_{1}|) + 9.$$ 
If $b(K_{1}) \leq 15$, then, by \cite[Theorem 3.4]{DHM} and its proof, we have 
$$b(H) \leq \log |H| / \log |V| + b(K_{1}) + 1 + (\log 48 / \log (2^{b(K_{1})})) \leq \log |H| /\log |V| + 17.$$ 

So assume now that $b(K_{1}) \geq 16$ and $b(K_{1}) \leq 2 (\log |K_{1}| / \log |V_{1}|) + 9$. In that case our proof strictly follows the arguments of \cite{DHM}, but in order to be able to prove Theorem \ref{ez} we need to give more precise
estimates for the constants appearing there. In the proof we will freely use the concepts and notation of \cite{DHM}.

A main step in proving \cite[Theorem 3.17]{DHM} was to reduce
the problem for linear groups which do not preserve any tensor product
decomposition of the vector space (possibly over a proper field
extension of the base field). In order for the reduction argument to work,  
a generalisation of the problem was needed. 

Let us view $H$ not just as a subgroup of $GL(V)$ but also as an abstract group. 
We will define certain maps $X:H \rightarrow GL(V)$. For this let $T_{V}$ denote the group 
\[T_V=\{g\in GL(V)\,|\,g(V_i)=V_i \textrm{ and }g|_{V_i}\in Z(GL(V_i))
\ \ \forall 1\leq i\leq t\}\simeq (\FF q^\times)^t\] where $\FF q$ is
the field of definition for $V$. According to \cite[Definition
3.5]{DHM} we say that a map $X:H \rightarrow GL(V)$ is a $\pmod
{T_{V}}$-representation of $H$ if the following two properties hold:
(1) $X(g)$ normalizes $T_{V}$ for every $g\in H$; and (2) $X(gh)T_{V}
= X(g)X(h)T_{V}$ for every $g,h\in H$. We will always consider $\pmod
{T_{V}}$-representations $X$ of $H$ such that the group $X(H)T_{V}$ acts
transitively on the set of factors $\Pi = \{ V_{1}, \ldots , V_{t} \}$
of the above direct sum decomposition of $V$. Let $b_{X}(H)$ denote
the minimal base size of the group $X(H)T_{V}$. 

In the rest of this subsection we will show that if $X$ is an imprimitive irreducible linear representation of $H$ preserving the decomposition $V = V_{1} \oplus \cdots \oplus V_{t}$, then $b_X(H) \leq 2 (\log |H| / \log |V|) + 17$, or $b(H) \leq 2 (\log |H| / \log |V|) + 17$. Theorem \ref{ez} will follow by specializing to the case when $X$ is the identity. 

If the representation of $H$ is alternating-induced in the sense of \cite[Section 3.2]{DHM}, then $b(H) \leq 2 (\log |H| / \log |V|) + 17$, by \cite[Theorem 3.9]{DHM}. 

By \cite[Section 3.4]{DHM} (and especially by \cite[Corollary 3.15]{DHM}), it is sufficient to establish the proposed bound in the claim for $b(H)$ or for $b_{X}(H)$ in case $X$ is a (not necessarily linear) $\pmod{T_{V}}$-representation of $H$ which is classical-induced satisfying the multiplicity-free condition, in the sense of \cite[Section 3.3]{DHM}. Let $X$ be such a representation of $H$. 

Let $N$ be the kernel of the action of $H$ on $\Pi$. Let $\fX : H \go N_{GL(V(p))}(T_V)/T_V$ be the homomorphism defined by $\fX(h):=X(h)T_V/T_V$ where $V(p)$ denotes the vector space $V$ defined over the prime field (of size $p$) of $\FF q$ and where $h \in H$. In \cite[Section 3.3]{DHM} a bound is given for $b_{X}(H)$. By the argument after \cite[Theorem 3.11]{DHM} (from the second to the fifth paragraphs), we get $b_{X}(H) < (\log |H| / \log |V|) +12$ when $\fX(N) = 1$. So assume that $\fX(N) \not= 1$. 

We use the notation (and a minor modification of the argument) of the paragraph following \cite[Theorem 3.11]{DHM}. In this case $\mathrm{Soc}(\fX (N))$ is a subdirect product of isomorphic simple classical groups $S_{1}, \ldots , S_t$. The linking factor, denoted by $r$, is at most $2$. We have $|N| \geq {|S_{1}|}^{t/r}$. Since $b(K_{1}) \geq 16$, the center of $K_{1}$ has size less than ${|V_{1}|}^{1/16}$ and $|\mathrm{Out}(S_{1})| \leq {|V_{1}|}^{6/16}$, (the latter by \cite[Lemma 7.8]{GMP}). Thus $|S_{1}| \geq |K_{1}|/{|V_{1}|}^{1/2}$, which implies $|N| \geq {|K_{1}|}^{t/r}/{|V|}^{1/2}$. 

Assume that $r=1$. By Theorem \ref{main}, we have 
\[
b(K_{1})\leq 2 (\log |K_{1}| / \log |V_{1}|) + 9
       = 2 (\log ({|K_{1}|}^{t}) / \log |V|) + 9
       \leq 2 (\log |N| / \log |V|) + 10.
\]
By \cite[Theorem 3.4]{DHM} and its proof, we then get $b(H) \leq 2
(\log |H| / \log |V|) + 12$ since $\log |V_{1}| \geq 16$.

Now assume that $r = 2$, so $t=2k$ for some integer $k$. By changing the order (if necessary) of
the summands in the direct sum $V=\oplus_{i=1}^t V_i$, we can assume that
the sets $\Delta_i:=\{V_{2i-1},V_{2i}\}$, for all $i$ with $1\leq i\leq k$, form a system of
$H$-blocks with $S_{\Delta_i} := \mathrm{Soc}(\fX_{\Delta_i}(N))$ a full 
diagonal subgroup of $S_{2i-1}\times S_{2i}$. Here $\fX_{\Delta_{i}}$ is defined as follows. Put $V_{\Delta_{i}}:= V_{2i-1} \oplus V_{2i}$ and let $T_{V_{\Delta_{i}}}$ be defined analogously as $T_{V}$. Let $X_{i} : N_{H}(\Delta_{i}) \rightarrow GL(V_{\Delta_{i}})$ be the $\pmod
{T_{V_{\Delta_{i}}}}$-representation of $N_{H}(\Delta_{i})$ obtained naturally from $X$ by restricting first from $H$ to $N_{H}(\Delta_{i})$ and then from $V$ to $V_{\Delta_{i}}$. Now define $\fX_{\Delta_{i}}$ to be the homomorphism given by $\fX_{\Delta_{i}}(h):= X_{i}(h) T_{V_{\Delta_{i}}} / T_{V_{\Delta_{i}}}$ for $h \in N_{H}(\Delta_{i})$. The multiplicity-free condition (and the definition of $\Delta_{i}$) guarantees that there are no 
functions $\varphi_i:V_{2i}\to V_{2i-1}$ and $\lambda_i:N_H(\Delta_i)
\to \FF q^\times$ such that $\varphi_i$ is a semilinear invertible map
and $$X_{2i-1}(g)=\lambda_i(g)\cdot \varphi_{i} X_{2i}(g) \varphi_i^{-1}$$ for every 
$g\in C_H(\Delta_i)$. By using \cite[Theorem 2.1.4]{KL}, it follows that 
$S_1\simeq PSL(V_1)$. Thus, $|N|\geq (q^{\dim^2 V_1-\dim V_1})^{t/2}/{|V|}^{1/2}$, 
which means that $$b_{X_{\Delta_1}}(N_H(\Delta_1))\leq \dim V_1\leq 
2(\log |N|/\log|V|) +2.$$ Now we can apply \cite[Theorem 3.4]{DHM} to the 
$H$-invariant direct sum decomposition $V = \oplus_{i=1}^{k} V_{\Delta_i}$ to deduce
that $b_X(H)\leq 2(\log |H|/\log |V|) +3+\log 48 < 2(\log |H|/\log |V|) +9$.

This completes the proof of Theorem \ref{ez}.

\section{Proof of Corollary \ref{maincorollary}}
\label{SecCor}

In this section we will prove Corollary \ref{maincorollary}. Let $G$ be a primitive permutation group of degree $n$. For later use, we recall the definition of {\it standard} actions of almost simple primitive groups: these occur for groups with socle an alternating group $\mathrm{Alt}(m)$ or a classical group. In the former case  they are actions on an orbit of subsets or partitions of $\{1,\ldots,m\}$; and in the latter, they are subspace actions (as defined in Section 3).

Assume first that $G$ is $\mathrm{Sym}(m)$ or $\mathrm{Alt}(m)$ for some integer $m \geq 5$. We consider standard actions of $G$.

If the action of $G$ is on a set of partitions of the underlying set of size $m$, then $b(G) \leq \log n + 4$ by  Theorem \ref{bcnprop}. The right hand side is less than $\sqrt{n}$ for $n \geq 256$, and $b(G) < 12$ otherwise. 

Now assume that $G$ acts on $k$-element subsets of $\{ 1, \ldots ,m \}$ for some integer $k$ with $2 \leq k \leq m/2$ and $n = \binom{m}{k}$. Assume also that $b(G) \geq 26$.

Let $k^{2} \leq m$. Then, by Theorem \ref{precise}, $b(G) \leq \frac{2m}{k+1} + 1$. Since $b(G) \geq 26$, we have $m \geq 38$ and $n \geq 625$. For $k =2$ and $n \geq 625$ we get $b(G) \leq \frac{2m}{3} + 1 < \sqrt{n}$. For $k \geq 3$ and $n \geq 625$ we find that $m^{2} \leq \binom{m}{k} = n$ and so $b(G) < \sqrt{n}$. 

Let $k^{2} > m$. Then $k \geq 3$ since $m \geq 5$. By Theorem \ref{precise}, $b(G) \leq \left\lceil\log_{\lceil t\rceil}(m)\right\rceil
\left(\lceil t\rceil-1\right)$ where $t = m/k$. Since $b(G) \geq 26$, this forces $m \geq 18$. In particular $k \geq 5$ and $n \geq 625$. Thus assume that $m \geq 18$, $k \geq 5$, and $n \geq 625$. Let $k \geq 8$. By Theorem \ref{precise} and the assumption $k^{2} > m$, we have 
$b(G) \leq (\log m + 1)t < (2 \log k + 1)t$. Since $t \geq 2$, it follows that $(2 \log k + 1)t < t^{k/2} = {(\frac{m}{k})}^{k/2} < \sqrt{\binom{m}{k}}$.
If $5 \leq k \leq 7$, then $b(G) \leq 6 (\log m + 1)$, by Theorem \ref{precise} (recall that $m<k^2$). The right hand side is less than $\sqrt{\binom{m}{k}}$ provided that $m \geq 18$ and $5 \leq k \leq m/2$. 

Next assume that $G$ is a group as in Section \ref{Sec4.3.1}. Let us use the notation and results of that section. By (\ref{e11}), we have $b(G) < \log k + 1 + n^{1/k}$. For $k \geq 3$ this is less than $\sqrt{n}$ since $n\geq 5^k \geq 125$. Now let $k = 2$. Then $G$ is a subgroup of $\mathrm{Sym}(t) \wr \mathrm{Sym}(2)$ with $n = t^{2}$. In this case $b(G) \leq t$, that is, $b(G) \leq \sqrt{n}$. 

At this point, by \cite[Theorem 1.1]{Mar}, we may assume that $|G| \leq n^{1 + \log n}$ (and $n \geq 26$).

If $G$ is as in Section \ref{Sec4.3.2} or as in Section \ref{Sec4.4}, then $n > 2500$, and so by Theorem \ref{mainresult}, $$b(G) \leq 2 \frac{\log |G|}{\log n} + 24 \leq 2 \log n + 26 < \sqrt{n}.$$ 

If $G$ is as in Section \ref{Sec4.2}, we use Fawcett's \cite{Fawcett} bound $b(G) \leq (\log |G|/\log n) + 3$ to obtain $b(G) \leq \log n + 4 < \sqrt{n}$ provided $n \not= 60$. If $n = 60$, then $|G| \leq 4 n^{2}$, and so $b(G) \leq 5$.  

Assume that $G$ is almost simple. 
If the action of $G$ is non-standard, then $b(G) \leq 7$ by \cite{Burness} and \cite{BurnessLiebeckShalev}. Thus assume that the action of $G$ is standard. In particular the socle of $G$ is a simple classical group or an alternating group. The case when the socle is an alternating group was treated before.  

Let $G$ be an almost simple group with socle a classical group with natural module a vector space of dimension $d$ over some finite field. Since $|G| \leq n^{1 + \log n}$ and $b(G) \leq 2 (\log |G| / \log n) + 16$, the latter by Theorem \ref{generalclassical}, we see that $b(G) < \sqrt{n}$ for $n \geq 1600$. Thus assume that $n < 1600$. By Theorem \ref{classicalmain} and Section \ref{Sec3.3}, we have $b(G) \leq d + 14$. By \cite[Table 5.2.A]{KL}, we find that $d$ must be at most $11$, and so $b(G) \leq 25$.  

Finally, assume that $G$ is of affine type with $n \geq 4$. Put $n = p^{d}$ where $p$ is a prime and $d$ is an integer. Then $b(G) \leq 1 + d$. This is at most $p^{d/2}$ unless $p = 2$ and $2 \leq d \leq 5$. In particular $b(G) \leq 6$ and $n \leq 32$. 

This completes the proof of Corollary \ref{maincorollary}.

\end{document}